\documentclass[11pt]{amsart}
\usepackage[toc,page]{appendix}
\usepackage{latexsym,amsfonts,amssymb,amsthm,amsmath,mathrsfs,color,amscd,graphicx,fullpage,fancyhdr,ulem} % ,parskip
\usepackage[colorlinks=true,urlcolor=blue,
citecolor=red,linkcolor=blue,linktocpage,pdfpagelabels,bookmarksnumbered,bookmarksopen]{hyperref}

\newcommand{\s}[1]{{\mathcal #1}}

\newcommand{\bb}[1]{{\mathbb #1}}

\newcommand{\T}{\mathbb{T}}
\newcommand{\R}{\mathbb{R}}
\newcommand{\Z}{\mathbb{Z}}

\def\dd{\,{\rm d}} % d in the integral
\newcommand{\ds}{\displaystyle}
\newcommand{\weakly}{\rightharpoonup}

\def\a{\alpha}

\def\d{\delta}

\def\l{\lambda}

\def\e{\varepsilon}

\newcommand{\cA}{{\mathcal A}}
\newcommand{\cB}{{\mathcal B}}

\newcommand{\sM}{{\mathscr M}}
\newcommand{\sP}{{\mathscr P}}

\newtheorem{theorem}{Theorem} %[section]
\newtheorem{corollary}[theorem]{Corollary}

\newtheorem{lemma}[theorem]{Lemma}
\newtheorem{proposition}[theorem]{Proposition}

\newtheorem{definition}[theorem]{Definition}

\newtheorem{remark}[theorem]{Remark}

\numberwithin{equation}{section}
\numberwithin{theorem}{section}

\title{Sobolev regularity for first order Mean Field Games}

\author[P.J. Graber]{P. Jameson Graber}
\address{Department of Mathematics, Baylor University, One Bear Place, Waco, TX 97328, USA}
\email{jameson\_graber@baylor.edu}

\author[A.R. M\'esz\'aros]{Alp\'ar R. M\'esz\'aros}  
\address{Department of Mathematics, UCLA, 520 Portola Plaza, Los Angeles, CA 90095, USA}
\email{alpar@math.ucla.edu}


\dedicatory{Version: \today}

\begin{document}
\maketitle

\begin{abstract}
In this paper we obtain Sobolev estimates for weak solutions of first order variational Mean Field Game systems with coupling terms that are local functions of the density variable. Under some coercivity conditions on the coupling, we obtain first order Sobolev estimates for the density variable, while under similar coercivity conditions on the Hamiltonian we obtain second order Sobolev estimates for the value function. These results are valid both for stationary and time-dependent problems. In the latter case the estimates are fully global in time, thus we resolve a question which was left open in \cite{ProSan}. Our methods apply to a large class of Hamiltonians and coupling functions.
\end{abstract}

\section{Introduction}
The theory of Mean Field Games (briefly MFG in the sequel) saw the light more than a decade ago, thanks to the works of J.-M. Lasry and P.-L. Lions on the one hand (see \cite{LasLio06i,LasLio06ii,LasLio07}) and M. Huang, R. Malham\'e and P. Caines, on the other hand (see \cite{HuaMalCai}). Their main motivation was to study limits of Nash equilibria of (stochastic or deterministic) differential games when the number of players tends to infinity. Since then, it became a very lively and active branch of the theory nonlinear partial differential equations.

In its most simple form an MFG can be described in an informal way as follows. As data, one considers the given quantities: $T>0$ (the time horizon), $H:\T^d\times\R^d\to\R$ (the Hamiltonian), $\phi_T:\T^d\to\R$ (the final cost of the agents), $f:\T^d\times\sP(\T^d)\to\R$ (the running cost of the agents) and $m_0\in\sP(\T^d)$ (the initial distribution of the agents), where $\T^d:=\R^d/\Z^d$ denotes the $d$-dimensional flat torus and $\sP(\T^d)$ stands for the set of nonnegative Borel probability measures on $\T^d$. A typical agent predicts the evolution of the agents' density $m:[0,T]\times\T^d\to[0,+\infty)$ and to find her/his optimal strategy, solves the control problem
\begin{equation}\label{prob:control}
\inf_{\a}\left\{\int_t^T L(\gamma(s),\a(s))+f(\gamma(s),m(s,\gamma(s)))\dd s +\phi_T(\gamma(T))\right\}=:\phi(t,x)
\end{equation}
subject to 
$$
\left\{
\begin{array}{ll}
\gamma'(s)=\alpha(s), & s\in(t,T]\\
\gamma(t)=x\in\T^d.
\end{array}
\right.
$$
Here the Lagrangian $L:\T^d\times\R^d\to\R$ is the Legendre-Fenchel transform of $H$ w.r.t. the second variable. The optimal strategy is given in feedback form, hence for the agent it is optimal to play  $-D_\xi H(\gamma(s),\nabla\phi(s,\gamma(s)))$. Having this velocity field as drift, the evolution of the agents' density is given by the solution of a continuity equation. If the prediction coincides with this evolution, the game has a Nash equilibrium. Thus, searching for Nash equilibria for MFG is equivalent to solving the following system of PDEs:

\begin{equation}\label{eq:MFG_intro}
\left\{
\begin{array}{ll}
-\partial_t\phi + H(x,\nabla\phi) = f(x,m), & {\rm{in}}\ (0,T)\times\T^d,\\[5pt]
\partial_t m -\nabla\cdot (D_\xi H(x,\nabla\phi)m)=0, & {\rm{in}}\ (0,T)\times\T^d,\\[5pt]
m(0,\cdot)=m_0,\ \ \phi(T,\cdot)=\phi_T, & {\rm{in}}\ \T^d.
\end{array}
\right.
\end{equation}

With a well-chosen time rescaling, one can introduce stationary MFG systems as long time average of time dependent ones. These take the form

\begin{equation}\label{eq:MFG_intro_stat}
\left\{
\begin{array}{ll}
 \l+  H(x,\nabla\phi) = f(x,m), & {\rm{in}}\ \T^d,\\[5pt]
 -\nabla\cdot (D_\xi H(x,\nabla\phi)m)=0, & {\rm{in}}\ \T^d,\\[5pt]
\ds \int_{\T^d}m\dd x=1, \; \ds\int_{\T^d}\phi\dd x=0, \; m\ge 0, & {\rm{a.e.\ in\ }}\T^d,
\end{array}
\right.
\end{equation}
where an additional variable $\lambda\in\R$ appears in the first equation, which plays the role of the ergodic constant. We refer to \cite{CarLasLioPor,CarLasLioPor2} and \cite[Section 4]{CarGra} for results on the limiting procedure and well-posedness of \eqref{eq:MFG_intro_stat}. 

The case when the running cost $f$ is monotone and regularizing (nonlocal) in the measure variable, is well understood in the literature. Already in the original works of J.-M. Lasry and P.-L. Lions (see also \cite{Lions-course}) has been shown the well-posedness of both systems \eqref{eq:MFG_intro} and \eqref{eq:MFG_intro_stat}. By contrast, when $f$ is a local function of $m$, the question of well-posedness of these systems is more subtle. 
Notice also that in the lack of sufficient regularity for the density variable $m$ (when $f$ is a local function of $m$), the control problem \eqref{prob:control} is not meaningful, while the system \eqref{eq:MFG_intro} may still have solutions in a suitable weak sense. 

In the case when $f$ is a local function of $m$, PDE techniques, which are used successfully for second order systems (see for instance \cite{GomPimSan15,GomPimSan16,GomPimVos,Por}), can no longer be used to show the well-posedness of the system \eqref{eq:MFG_intro}. Also, in general one cannot hope for the existence of classical solutions of \eqref{eq:MFG_intro}. Nevertheless, as it was already mentioned  in \cite{LasLio06ii} that when $f$ is non-decreasing in its second variable and $H$ convex in its second variable, at least formally, systems \eqref{eq:MFG_intro} and \eqref{eq:MFG_intro_stat} can be seen as first order optimality conditions of two convex optimization problems in duality. In the first order case (or when a degenerate diffusion is present) these arguments were made rigorous in a series of papers (see \cite{Car15,CarGra, CarGraPorTon}), and the existence of weak solutions (understood in a suitable sense) was shown. An important and interesting question in this context is the regularity of solutions of \eqref{eq:MFG_intro} and \eqref{eq:MFG_intro_stat}. 

A first result in this direction was obtained in \cite{CarPorTon}, which, it must be emphasized, provides first and foremost a regularity result for viscosity solutions of first-order Hamilton-Jacobi equations.
Using techniques involving the inverse H\"older inequality, the authors obtained $W^{1,1}_{{\rm{loc}}}((0,T)\times\T^d)$ estimates and a.e.~differentiability of $\phi$, depending only on the summability of the right-hand side of the equation.
In the context of mean field games, this amounts to requiring a sufficiently high growth condition on $f(x,m)$.
Notice that their estimates are not owing to the variational structure of the mean field games system; in particular the density variable plays no role.

\vspace{10pt}

A second direction in the search for Sobolev estimates for first order MFGs was initiated in \cite{ProSan} (see also \cite{San17}), where the authors obtained $H^1_{{\rm{loc}}}((0,T]\times\T^d)$ estimates for a well-chosen function of $m$, where $(\phi,m)$ is a weak solution of \eqref{eq:MFG_intro}. Their technique is based on a so-called {\it regularity by duality method}, which relies on the fact that \eqref{eq:MFG_intro} is the optimality condition for two convex optimization problems in duality. In fact, these techniques are inspired from \cite{CarMesSan} and more precisely they have their roots in \cite{Bre99} (see also \cite{AmbFig2}), where Y. Brenier introduced a similar approach to obtain regularity estimates for the pressure field arising from generalized solutions of the incompressible Euler equations. In a nutshell, the basic idea of this method is the following. One considers special competitors for the optimization problem involving the $m$ variable that are translations of the optimizers in time and space, then one compares the energy of these new competitors to the energy of the optimizers. Finally, assuming a sort of coercivity/monotonicity property of the running cost $f$ in the second variable (see the precise assumption in Sections \ref{sec:prelim}-\ref{sec:stationary}), one can obtain a differential quotient estimate for a function of $m$ via the difference of the energies and deduce the Sobolev estimate. A drawback of this method is that since the optimization problem is subject to the continuity equation with an initial condition, when one constructs the new competitors (via the time-space translations), one must preserve the initial condition, so it seems impossible to construct translations that allow to vary the initial time as well. We shall give more details on this method in Section \ref{sec:prelim}.

Let us mention that in the framework of first order models, P.-L. Lions in his lectures (\cite{Lions-course}) showed -- at least formally -- how to obtain a priori estimates in the case of the {\it planning problem} (where the initial and final density of the agents is prescribed). In particular, by rewriting the system as a very degenerate elliptic system in time and space, he obtains global $L^\infty$ estimates for the density variable $m$ and $W^{1,\infty}$ estimates for the value function $\phi$. However, his techniques require strong smoothness assumptions on the prescribed densities, Hamiltonian and coupling and positivity (and monotonicity) of the coupling function. Nevertheless, our objective in this paper is different: it is to obtain first order Sobolev estimates for the $m$ variable and second order Sobolev estimates for $\phi$. Thus, it seems that such an approach using degenerate elliptic equations to obtain these regularity estimates cannot be applied in our case. In the same context enters the recent paper \cite{LavSan}, where the authors obtain local in time $L^\infty$ estimates for first order MFG models (with quadratic Hamiltonians), using variational techniques. Their approach, however, is very different from the one of Lions and from ours as well.

Actually, the same techniques as in \cite{ProSan,San17} can be used -- more or less in a straight forward manner -- for stationary problems (since there is no time involved there) as well, so as our preparatory results, we present how to gain Sobolev estimates for the solutions of \eqref{eq:MFG_intro_stat}. More precisely, we have

\begin{theorem}[Theorem \ref{thm:stat_main}, Theorem \ref{thm:stat_main2}]\label{thm:intro_stat}
We assume the standard regularity and growth conditions on the data $f$ and $H$ (see Sections \ref{sec:prelim} and \ref{sec:stationary} for the precise assumptions). Let $(\phi,m,\l)$ be a solution of \eqref{eq:MFG_intro_stat}.  Then, one has a second order Sobolev estimate for $\phi$, i.e.
$$J_* (H(\cdot,\nabla\phi)+\l)\in H^1(\T^d),$$
and a first order Sobolev estimate for $m$, i.e.
$$J(m)\in H^1(\T^d),$$
where $J,J_*:\R\to\R$ are functions depending on $f$ with a precise growth condition that depends on the growth condition of $f$ in its second variable.
\end{theorem}

As our main result, we obtain global in time Sobolev estimates for weak solutions of first order MFG systems up to the initial (and final) time. While we keep the variational framework and rely on the convex duality, because of the obstruction described above, our approach is different from the one used in \cite{ProSan}. Our main results can be summarized in an informal way in the theorem below.

\begin{theorem}[Proposition \ref{prop:space-regularity}, Corollary \ref{cor:time-reg}] \label{thm:intro_time}
We assume that $m_0,\phi_T \in W^{2,\infty}(\bb{T}^d)$ and we assume some regularity and growth conditions on $H$ (which has a growth like $|\cdot|^r$ for some $r>1$ in its second variable) and $f$ (which has a growth like $|\cdot|^{q-1}$ for some $q>1$ in its second variable). We refer to Sections \ref{sec:prelim} and \ref{sec:time} for the precise assumptions.

Then there exist a constant $C>0$ depending only on the data such that
\begin{itemize}
\item[(i)] $\|m^{\frac{q}{2} - 1}\nabla m\|_{L^2([0,T]\times\T^d)} \leq C;$ \vspace{5pt}
\item[(ii)] $\|m^{1/2}D (j_1(\nabla \phi))\|_{L^2([0,T]\times\T^d)} \leq C;$ \vspace{5pt}
\item[(iii)] if $r=2$, then $\|\partial_t (m^{q/2})\|_{L^1([0,T]\times\T^d)} \leq C.$
\end{itemize}
Here $j_1:\R^d\to\R^d$ is a function depending on $H$ and has a growth like $|\cdot|^{r/2}.$ 
\end{theorem}

The core idea of our technique is the following. The space translations of the optimizers will actually solve some auxiliary optimization problems, where the data are the space translations of the data of the original problems. This observation leads us to obtain fully global in time differential quotient estimates in the space variable. Then using the continuity equation -- when the Hamiltonian has a quadratic growth in its second variable -- the first order Sobolev estimates in time for $m$, globally in time, will be a consequence. 

\vspace{10pt}

Let us mention that while in \cite{ProSan} the authors considered only the case when $H(x,\xi)=\frac12|\xi|^2$, most of the computations seem to be adaptable to more general Hamiltonians to obtain {\it local in time} Sobolev estimates, but to gain estimates up to the final time $t=T$, it seems crucial for them that $D_\xi H$ is a Lipschitz function. In comparison with this, our techniques allow us to consider quite general Hamiltonians (with no restriction on the growth condition) to obtain the global in time estimates for the space variable. Nevertheless, for the global in time estimates in the time variable, it seems that we have a similar obstruction. 

Let us remark also that the estimate Theorem \ref{thm:intro_time}-(ii) seems to appear naturally in the context of Hamilton-Jacobi equations in connection with the nonlinear adjoint method introduced by L.C. Evans in \cite{Eva}. Nevertheless, we believe that it is interesting to see this kind of second order estimates in the framework of first order Hamilton-Jacobi equations arising from MFGs, especially if we take into consideration that solutions to these equations a priori had to be understood in a very weak sense.

As a last remark, let us mention that it seems that our techniques to obtain global in time Sobolev estimates can be applied for other type of problems as well, in particular in the case of the first order {\it planning problem} (see \cite{Lions-course}). These estimates -- together with other results on this kind of systems -- is the subject of a subsequent work (\cite{GraMes2}). 
\vspace{10pt}

The structure of the paper is simple. In Section \ref{sec:prelim} we collect the existing results on the well-posedness of first order variational MFG systems. The main assumptions, the convex duality results for the optimizations problems are also presented here. We also give a short summary on the main ideas of the technique ``regularity by duality'' presented in \cite{ProSan} in the framework of MFG.

In Section \ref{sec:stationary} we present the Sobolev estimates for first order stationary MFG systems and we prove Theorem \ref{thm:intro_stat}.

Finally, Section \ref{sec:time} is devoted to the proof of the global in time Sobolev estimates presented in Theorem \ref{thm:intro_time}.

\section{Preliminary results and the regularity by duality method}\label{sec:prelim}

We list the standing assumptions which are sufficient to ensure the well-posedness of both  \eqref{eq:MFG_intro} and \eqref{eq:MFG_intro_stat}.

\vspace{10pt}

{\bf Standing assumptions: } the following conditions on the data were used to get the result in \cite{CarGra}.

\begin{enumerate}
	\item[(H1)] (Conditions on the Hamiltonian) $H:\bb{T}^d \times \bb{R}^d \to \bb{R}$ is continuous in both variables, convex and differentiable in the second variable, with $D_\xi H$ continuous in both variables. Moreover, $H$ has superlinear growth in the gradient variable: there exist $r >1 $ and $C >0$ such that
	\begin{equation}
	\label{eq:hamiltonian_bounds}
	\frac{1}{rC}|\xi|^r-C \leq H(x,\xi) \leq \frac{C}{r}|\xi|^r + C,\ \ \forall\ (x,\xi)\in\T^d\times\R^d.
	\end{equation}
	We denote by $H^*(x,\cdot)$ the Fenchel conjugate of $H(x,\cdot)$, which, due to the above assumptions, satisfies
	\begin{equation}
	\label{eq:hamiltonian_conjugate_bounds}
	\frac{1}{r'C}|\zeta|^{r'}-C \leq H^*(x,\zeta) \leq \frac{C}{r'}|\zeta|^{r'} + C,\ \ \forall\ (x,\zeta)\in\T^d\times\R^d,
	\end{equation}
	where $r' = r/(r-1)$ is the conjugate of $r$. 
	\item[(H2)] (Conditions on the coupling) Let $f$ be continuous on $\bb{T}^d \times (0,\infty)$, strictly increasing in the second variable, and there exists a constant $C>0$ satisfying
	\begin{equation} \label{eq:coupling_growth}
	\frac{1}{C}|m|^{q-1} - C \leq f(x,m) \leq C|m|^{q-1} + C, ~~ \forall ~ m \geq 1, \forall\ x\in\T^d.
	\end{equation}
	\item[(H3)] (Conditions on the antiderivative of $f$) We define $F$ so that $F(x,\cdot)$ is an antiderivative of $f(x,\cdot)$ on $(0,\infty)$, that is,
\begin{equation}
F(x,m) = \int_1^m f(x,s)\dd s, ~~ \forall ~ m > 0.
\end{equation}
It follows that $F$ is continuous on $\bb{T}^d \times (0,\infty)$, is strictly convex and differentiable in the second variable, and satisfies the growth condition
\begin{equation} \label{eq:cost_growth}
\frac{1}{qC}|m|^q - C \leq F(x,m) \leq \frac{C}{q}|m|^q + C,~~~ \forall ~ m \geq 1,\ \forall\ x\in\T^d.
\end{equation}
For $m < 0$ we set $F(x,m) = +\infty$. 
We denote by $F(x,0)$ the limit $\lim_{m \to 0^+} F(x,m)$, which may be finite or $+\infty$.

We will denote throughout the conjugate exponent of $q$ by $p = q'$. Define $F^*(x,\cdot)$ to be the Fenchel conjugate of $F(x,\cdot)$ for each $x\in\T^d$. Note that
\begin{equation} \label{eq:cost_growth_star}
\frac{1}{pC}|a|^p - C \leq F^*(x,a) \leq \frac{C}{p}|a|^p + C, ~~ \forall a \geq 0,\ \forall\ x\in\T^d.
\end{equation}
\end{enumerate}

\subsection{Preliminary results on the stationary system}
Let us recall here the stationary MFG system.

\begin{equation}
\label{MFGergo}
\left\{\begin{array}{cll}
{\rm{(i)}}&  \lambda +H(x,D \phi) =f(x, m(x)), &{\rm{in}}\ \T^d,\\[5pt]
{\rm{(ii)}} & -\nabla\cdot (mD_\xi H(x, \nabla \phi))=0, &{\rm{in}}\ \T^d, \\[5pt]
{\rm{(iii)}}&  \ds\int_{\bb{T}^d}  m\dd x= 1; \; \ds\int_{\T^d}\phi\dd x=0;\; & m\geq 0,\ \ {\rm{a.e\ in\ }} \T^d.
\end{array}\right.
\end{equation}

\begin{definition}\label{def:weaksolergoMFG} We say that a triple $(\lambda,  \phi,m )\in \bb{R} \times W^{1, pr}(\bb{T}^d)\times L^q(\bb{T}^d)$ is a weak solution of \eqref{MFGergo} if
	\begin{itemize}
		\item[(i)] $m\geq 0$ a.e. in $\T^d$, $\displaystyle \int_{\bb{T}^d}m\dd x=1$, $\ds\int_{\T^d}\phi\dd x=0$ and $mD_\xi H(\cdot,\nabla \phi)\in L^{1}(\bb{T}^d)$, 
		\item[(ii)] Equation \eqref{MFGergo}-(i) holds in the following sense:
		\begin{equation}
		\label{eq:aeergo}
		\displaystyle \quad \lambda +H(x,\nabla \phi(x))= f(x,m(x)) \quad \; \mbox{\rm a.e. in $\{m>0\}$}
		\end{equation}
		and 
		\begin{equation}
		\label{eq:distribergo}
		\quad \lambda +H(x,\nabla \phi(x))\leq  f(x,m) \quad \mbox{\rm a.e. in}\; \bb{T}^d,  
		\end{equation}
		
		\item[(iii)] Equation \eqref{MFGergo}-(ii) holds:  
		\begin{equation}
		\label{eqcontdefergo}
		\displaystyle \quad -\nabla\cdot( mD_\xi H(x,\nabla \phi))= 0\quad {\rm in }\; \bb{T}^d, 
		\end{equation}
		in the sense of distributions.
	\end{itemize}
\end{definition}

In the sequel we summarize the results on the well-posedness of system \eqref{MFGergo}. We define, on $\overline{ \s{K}}_0:= \bb{R}\times W^{1,pr}(\bb{T}^d)$, the functional 
\begin{equation}
\label{DefmathcalA}
{\mathcal A}(\lambda,\phi)= \int_{\bb{T}^d}  F^*\left(x,\lambda+H(x,\nabla \phi(x)) \right) \dd x - \lambda. 
\end{equation}
Our first optimization problem is 
\begin{equation}\label{PB:dual2ergo}
\inf_{(\lambda,\phi)\in \overline{ \s{K}}_0} \mathcal A(\lambda,\phi)
\end{equation}
To describe the second optimization problem, let us denote by $\overline{\s{K}}_1$ the set of pairs $(m, w)\in L^1(\bb{T}^d) \times L^1(\bb{T}^d;\bb{R}^d)$ such that $m\geq 0$ a.e. in $\T^d$, $\displaystyle \int_{\bb{T}^d}m(x)\dd x=1$, and $w$ is divergence free, i.e. it satisfies in the sense of distributions
\begin{equation}
\label{conteqergo}
\nabla\cdot w=0,\ \  {\rm in}\ \  \bb{T}^d. 
\end{equation}
We define on $\overline{\s{K}}_1$ the functional
$$
{\mathcal B}(m,w)= \int_{\bb{T}^d} m(x) H^*\left(x, -\frac{w(x)}{m(x)}\right)+ F(x,m(x)) \dd x ,
$$
where we use the usual convention to define $H^*(x,-w/m)$, i.e.
\begin{equation}\label{conventionH*}
mH^*\left(x,-\frac{w}{m}\right) = \left\{ \begin{array}{ll} +\infty, & \mbox{\rm if $m=0$ and $w\neq 0$,} \\ 0, & \mbox{\rm if $m=0$ and $w = 0$.} \end{array} \right.
\end{equation}

Since $H^*$ and $F$ are bounded below and $m\geq 0$ a.e., the integral in ${\mathcal B}(m,w)$  is well defined in  $\bb{R}\cup\{+\infty\}$. 
The second optimal control problem is the following: 
\begin{equation}
\label{Pb:mw2ergo}
\inf_{(m,w)\in \overline{\s{K}}_1} \mathcal B(m,w)\;.
\end{equation}

\begin{lemma}[\cite{CarGra}]\label{Lem:dualiteergo} We have
	\begin{equation}
	\label{minmax}
	\min_{(\lambda,\phi)\in \overline{ \s{K}}_0}{\mathcal A}(\lambda,\phi) = - \min_{(m,w)\in \overline{\s{K}}_1} {\mathcal B}(m,w).
	\end{equation}
	Moreover, the minimum in the right-hand side is achieved by a unique pair $(m,w)\in \overline{\s{K}}_1$ satisfying $(m,w)\in  L^q( \bb{T}^d)\times L^{\frac{r'q}{r'+q-1}}( \bb{T}^d)$. 
\end{lemma}

\begin{theorem}[\cite{CarGra}] \label{theo:mainexergo} There exists at least one  solution $(\lambda,\phi,m)$ to the stationary MFG system \eqref{MFGergo}. Moreover, the pair $(\lambda,m)$ is unique. 

If $(m,w)\in \overline{\s{K}}_1$ is a minimizer of \eqref{Pb:mw2ergo} and $(\lambda, \phi)\in \overline{ \s{K}}_0$ is a minimizer of \eqref{PB:dual2ergo}, then $(\l,\phi,m)$ is a solution of the MFG system \eqref{MFGergo} and $w= -mD_\xi H(\cdot,\nabla \phi)$ a.e.. 

Conversely, any  solution $(\lambda,\phi, m)$ of \eqref{MFGergo}  is such that the pair $(m,-mD_\xi H(\cdot,\nabla \phi))$ is the minimizer of \eqref{Pb:mw2ergo} while $(\l,\phi)$ is a minimizer of \eqref{DefmathcalA}. 
\end{theorem}

\subsection{Preliminary results on the time dependent system}
Let us recall the time dependent MFG system that we will study.

\begin{equation} \label{eq:mfg}\left\{
\begin{array}{cll}
{\rm{(i)}} & -\partial_t\phi + H(x,\nabla \phi) = f(x,m), & {\rm{in\ }} (0,T)\times\T^d,  \\
(ii) & \partial_t m - \nabla\cdot \left(mD_\xi H(x,\nabla \phi)\right) = 0, & {\rm{in\ }} (0,T)\times\T^d, \\
(iii) & \phi(T,x) = \phi_T(x), m(0,x) = m_0(x), & {\rm{in\ }} \T^d.
\end{array}\right.
\end{equation}
First, we need to impose some additional standard assumption on the initial and final data.

\begin{enumerate}
\item[(H4)] (Conditions on the initial and final conditions) $m_0$ is a probability measure on $\bb{T}^d$ which is absolutely continuous with respect to Lebesgue measure, having density (which we also call $m_0$) in $C(\bb{T}^d)$. We suppose moreover that $m_0>0$ on $\bb{T}^d$. 
	We  assume that $\phi_T:\bb{T}^d\to \bb{R}$ is a Lipschitz continuous function on $\bb{T}^d$.
\end{enumerate}

\begin{definition} \label{def:weak}
	A pair $(\phi,m) \in BV((0,T) \times \bb{T}^d) \times L^q((0,T) \times \bb{T}^d)$ is called a weak solution to the system (\ref{eq:mfg}) if it satisfies the following conditions.
	\begin{enumerate}
		\item $\nabla \phi\in L^r([0,T]\times\T^d)$ and  the maps $mf(\cdot, m)$, $mH^*\left(\cdot,-D_\xi H(\cdot,\nabla \phi)\right) $ and   $mD_\xi H(\cdot,\nabla \phi)$ are integrable, 
		\item $\phi$ satisfies  a first-order Hamilton-Jacobi inequality
		\begin{equation} \label{eq:hjb_weak}
		-\partial_t \phi + H(x,\nabla \phi) \leq  f(x,m) 
		\end{equation}
		in the sense of distributions (tested against smooth non-negative functions), the boundary condition $\phi(T,\cdot)\leq \phi_T$ in the sense of trace and the following equality
		\begin{multline}
		\int_0^T \int_{\bb{T}^d}m(t,x) \left[ H(x,\nabla \phi(t,x))-\nabla \phi(t,x)\cdot D_\xi H(x, \nabla \phi(t,x)) -f(x,m(t,x))\right]\dd x\dd t \\
		= \int_{\bb{T}^d} (\phi_T(x)m(T,x))-\phi(0,x)m_0(x))\dd x \label{eq:ibp_weak}
		\end{multline}
		
		\item $m$ satisfies the continuity equation
		\begin{equation} \label{eq:continuity_weak}
		\partial_t m - \nabla\cdot(mD_\xi H(x,\nabla \phi)) = 0 ~~~\ {\rm{in}}~~(0,T) \times \bb{T}^d, ~~~ m(0,\cdot) = m_0,
		\end{equation}
		in the sense of distributions.
	\end{enumerate}
\end{definition}

We consider two optimal control problems which are in duality, see \cite{Car15}. By the abuse of notion, we use the same notations $\cA$ and $\cB$ for the two functionals as in the stationary case.
First, the control of the continuity equation reads as: minimize
\begin{equation}
\s{B}(m,w) = \int_0^T \int_{\bb{T}^d} mH^*\left(x,-\frac{w}{m}\right) + F(x,m) \dd x \dd t + \int_{\bb{T}^d} \phi_T m(T) \dd x
\end{equation}
over $\s{K}_1 = \{(m,w) \in L^1(\bb{T}^d) \times L^1(\bb{T}^d;\bb{R}^d) : \partial_t m + \nabla \cdot w = 0, \ m(0) = m_0 \}$, where the continuity equation holds in the sense of distributions.
(As usual, for $(m,w) \in \s{K}_1$, the measure-valued function $[0,T]\ni t\mapsto m(t)\in\sP(\T^d)$ is continuous in time in the sense of weak (narrow) convergence of probability measures, cf.~\cite{AmbCri}).

Second, we control the Hamilton-Jacobi equation: minimize
\begin{equation}
\s{A}(\phi) = \int_0^T \int_{\bb{T}^d} F^*(x,-\partial_t\phi + H(x,\nabla \phi))\dd x \dd t - \int_{\bb{T}^d} \phi(0)m_0 \dd x
\end{equation}
over the set $\s{K}_0$ consisting of maps $\phi \in C^1([0,T] \times \bb{T}^d)$ such that $\phi(T,\cdot) = \phi_T$.

Set $E_1 = C^1([0,T] \times \bb{T}^d), E_0 = C([0,T] \times \bb{T}^d;\bb{R}) \times C([0,T] \times \bb{T}^d;\bb{R}^d)$.
On $E_1$ we define
\begin{equation}
\s{F}(\phi) = -\int_{\bb{T}^d} \phi(0)m_0 + \chi_{\s{K}_0}(\phi).
\end{equation}
On $E_0$ define
\begin{equation}
\s{G}(a,b) = \int_0^T \int_{\bb{T}^d} F^*(x,-a(t,x) + H(x,b(t,x))) \dd x \dd t.
\end{equation}
Then $\s{A}$ can be written $\s{A}(\phi) = \s{F}(\phi) + \s{G}(\Lambda(\phi))$ where $\Lambda : E_1 \to E_0$ is given by $\Lambda(\phi) = (\partial_t \phi,\nabla \phi)$.

The two problems are in duality:
\begin{theorem}[cf.~\cite{Car15}] \label{thm:duality}
	
	We have
	\begin{equation} \label{eq:duality}
	\inf_{\phi \in \s{K}_0} \s{A}(\phi) = -\min_{(m,w) \in \s{K}_1} \s{B}(m,w).
	\end{equation}
	Moreover, the minimum on the right-hand side is achieved by a unique pair $(m,w) \in \s{K}_1$ which must satisfy $m \in L^q([0,T] \times \bb{T}^d)$ and $w \in L^{\frac{r'q}{r'+q-1}}([0,T] \times \bb{T}^d)$.
\end{theorem}

\begin{proof}
	Use the Fenchel-Rockafellar Theorem, which shows that
	$$
	\inf_\phi \s{F}(\phi) + \s{G}(\Lambda(\phi)) = \max_{(m,w)} -\s{G}^*(-(m,w)) - \s{F}^*(\Lambda^*(m,w)).
	$$
	The right-hand side is just $-\min \s{B}$.
	See \cite{Car15} for details.
\end{proof}

We can relax the second problem in the following way.
$\s{K}$ will be defined as the set of all pairs $(\phi,\alpha) \in BV([0,T]\times\T^d) \times L^1([0,T]\times\T^d)$ such that $\nabla \phi \in L^r([0,T] \times \bb{T}^d)$,
$\phi(T,\cdot) \leq  \phi_T$ in the sense of traces, $\alpha_+ \in L^p((0,T) \times \bb{T}^d)$, $\phi \in L^\infty([t,T]\times\bb{T}^d)$ for every $t \in (0,T)$, and
\begin{equation} \label{eq:inequality_in_distribution}
-\partial_t \phi + H(x,\nabla \phi) \leq \alpha.
\end{equation}
in the sense of distributions.
Then on $\s{K}$ we define the functional
$$
\s{A}(\phi,\alpha) = \int_0^T \int_{\bb{T}^d} F^*(x,\alpha(t,x))\dd x \dd t - \int_{\bb{T}^d} \phi(0,x)m_0(x)\dd x.
$$
We have the following, due to \cite{CarGra}.
\begin{theorem}
	$\ds\inf_{\phi \in \s{K}_0} \s{A}(\phi) = \min_{(\phi,\alpha) \in \s{K}} \s{A}(\phi,\alpha)$, the latter attained by some $(\phi,\alpha) \in \s{K}$ of which $\phi$ can be obtained as a limit of smooth functions $\phi_n$.
\end{theorem}

Minimizers of $\s{A}(\phi,\alpha)$ and $\s{B}(m,w)$ precisely characterize weak solutions to the MFG system.

\begin{theorem}[Existence and (partial) uniqueness of weak solutions] \label{thm:minimizers_weak}
	${\rm{(i)}}$ If $(m,w) \in \s{K}_1$ is a minimizer of $\s{B}$ and $(\phi,\alpha) \in {\s{K}}$ is a minimizer of $\s{A}$, then $(\phi,m)$ is a weak solution of (\ref{eq:mfg}) and $\alpha(t,x) = f(x,m(t,x))$ almost everywhere. 
	
	${\rm{(ii)}}$ Conversely, if $(\phi,m)$ is a weak solution of (\ref{eq:mfg}), then there exist functions $w, \alpha$ such that $(\phi,\alpha) \in \s{K}$ is a minimizer of $\s{A}$ and $(m,w) \in \s{K}_1$ is a minimizer of $\s{B}$.
	
	${\rm{(iii)}}$ If $(\phi,m)$ and $(\phi',m')$ are both weak solutions to (\ref{eq:mfg}), then $m = m'$ almost everywhere while $\phi = \phi'$ almost everywhere in the set $\{m > 0\}$.
\end{theorem}

Before we continue, let us remark that in \cite{CarGra}, Theorem \ref{thm:minimizers_weak} holds only under the extra assumption $r > d(q-1)$.
The reason for this assumption was to have a priori upper bounds on the solution $\phi$.
However, in \cite{CarGraPorTon} it was shown that this is not necessary; one has in general a priori bounds on the positive part $\phi_+$ in $L^\eta$ for some $\eta > 1$.
In particular, we have
\begin{theorem}[cf. Theorem 3.3 in \cite{CarGraPorTon}]\label{thm:a priori bounds for hj}
	Suppose $\phi \in L^r((0,T);W^{1,r}(\bb{T}^d))$ satisfies, in the sense of distributions,
	\begin{equation}
	\left\{
	\begin{array}{lrl}
	{\rm{(i)}} & -\partial_t \phi + c_0|\nabla \phi|^r \leq \alpha, & {\rm{in\ }}(0,T)\times\T^d,\\[5pt]
	{\rm{(ii)}} & \phi(T,x) \leq \phi_T(x), & {\rm{a.e.\ in\ }}\T^d,
	\end{array}\right.
	\end{equation}
	for some $r > 1, c_0 > 0, \alpha \in L^p([0,T] \times \bb{T}^d)$ and $\phi_T \in L^\infty(\bb{T}^d)$.
	Then there exists a constant $C = C(p,d,r,c_0,T,\|\alpha\|_{L^p([0,T] \times \bb{T}^d)},\|\phi_T\|_{L^\eta(\bb{T}^d)})$ such that
	\begin{equation}
	\|\phi_+\|_{L^\infty([0,T];L^\eta(\bb{T}^d))} + \|\phi_+\|_{L^\gamma([0,T] \times \bb{T}^d)} \leq C
	\end{equation}
	where $\eta = \frac{d(r(p-1)+1)}{d-r(p-1)}$ and $\gamma = \frac{rp(1+d)}{d-r(p-1)}$ if $p < 1 + \frac{d}{r}$ and $\eta = \gamma = +\infty$ if $p > 1 + \frac{d}{r}$.
\end{theorem}
Using Theorem \ref{thm:a priori bounds for hj}, one can modify the proof of Theorem 2.9 in \cite{CarGra} to get existence of minimizers for the relaxed problem $\ds\min_{(\phi,\alpha) \in \s{K}} \s{A}(\phi,\alpha)$.
Indeed, let $\phi_n \in \s{K}_0$ be a minimizing sequence, so that
\begin{equation}
\int_0^T \int_{\bb{T}^d} F^*(x,\alpha_n)\dd x \dd t - \int_{\bb{T}^d} \phi_n(0)m_0 \dd x \to \inf_{\phi \in \s{K}_0} \s{A}(\phi)
\end{equation}
where $\alpha_n = -\partial_t\phi_n + H(x,\nabla\phi_n)$. For $\e\in\R$, 
set $K_\e = \sup_{x \in \bb{T}^d} F(x,\e)$ and note that $F^*(x,a) \geq a\e - K_\e$ for all $a\in\R$.
By the bounds \eqref{eq:cost_growth_star} it follows that
\begin{equation} \label{eq:minimizer-bounds1}
\int_0^T \int_{\bb{T}^d} F^*(x,\alpha_n)\dd x \dd t
\geq \frac{1}{pC}\|(\alpha_n)_+\|_{L^p}^p - \e \int_0^T \int_{\bb{T}^d} (\alpha_n)_-\dd x \dd t - K_\e T - C.
\end{equation}
Multiplying $(\alpha_n)_- = (\alpha_n)_+ + \partial_t \phi_n - H(x,\nabla \phi_n)$ by $m_0$ then integrating, using \eqref{eq:hamiltonian_bounds} and the fact that $m_0$ is bounded above and below by positive constants, we deduce
\begin{multline} \label{eq:minimizers-bounds2}
\int_0^T \int_{\bb{T}^d} (\alpha_n)_-\dd x \dd t
\leq C\int_0^T \int_{\bb{T}^d}(\alpha_n)_+ m_0 \dd x \dd t + \int_{\bb{T}^d} \phi_T m_0 - \phi_n(0)m_0 \dd x + CT
\\
 \leq C\|(\alpha_n)_+\|_{L^p} - \int_{\bb{T}^d} \phi_n(0)m_0 \dd x + C.
\end{multline}
Since
\begin{equation}
\int_0^T \int_{\bb{T}^d} F^*(x,\alpha_n)\dd x \dd t - \int_{\bb{T}^d} \phi_n(0)m_0 \dd x \leq C
\end{equation}
we can combine \eqref{eq:minimizer-bounds1} and \eqref{eq:minimizers-bounds2} to get
\begin{equation}
\frac{1}{pC}\|(\alpha_n)_+\|_{L^p}^p - \e C\|(\alpha_n)_+\|_{L^p} - (1 - \e) \int_{\bb{T}^d} \phi_n(0)m_0 \dd x 
\leq C + K_\e T.
\end{equation}
By fixing $\e > 0$ small enough and using \eqref{eq:minimizers-bounds2} once more, we obtain a priori upper bounds on both $\|(\alpha_n)_+\|_p$ and $-\int_{\bb{T}^d} \phi_n(0)m_0 \dd x$.
On the other hand, since
\begin{multline}
-\int_{\bb{T}^d} \phi_n(t)m_0 \dd x
\leq
-\int_{\bb{T}^d} \phi_n(0)m_0 \dd x
+ \int_0^t \int_{\bb{T}^d} (\alpha_n - H(x,\nabla \phi_n))m_0 \dd x
\\
\leq -\int_{\bb{T}^d} \phi_n(0)m_0 \dd x + C\|(\alpha_n)_+\|_p + C,
\end{multline}
we also have an upper bound on $-\int_{\bb{T}^d} \phi_n(t)m_0 \dd x$ that holds uniformly in time.
Combining this with Theorem \ref{thm:a priori bounds for hj}, we deduce that $\phi_n$ is bounded in $L^\infty([0,T];L^1(\bb{T}^d))$.
Moreover, \eqref{eq:minimizers-bounds2} implies that $(\alpha_n)_-$ is bounded in $L^1$.

The rest of the proof now follows that of Theorem 2.9 in \cite{CarGra} and Proposition 5.4 in \cite{CarGraPorTon}.
In particular, a modified version of $\alpha_n$ now has a weak limit $\alpha$ in $L^1([0,T]\times\T^d)$, and $\phi_n$ is seen to have a weak limit $\phi$ in $BV([0,T]\times\T^d)$, and one shows (cf.~\cite{CarGraPorTon}) that $(\phi,\alpha) \in \s{K}$ is a minimizer of $\s{A}$.

The above arguments imply in particular that both optimization problems, and hence the MFG system have a solution, without imposing the joint assumption on the growth condition of $f$ and $H$.

\subsection{The regularity by duality method}

The idea to obtain Sobolev estimates by duality on the $m$ variable in the system \eqref{eq:MFG_intro} used in \cite{ProSan} dates back to \cite{Bre99} (see also \cite{AmbFig2} and \cite{CarMesSan}) and it is as follows. Consider the two optimization problems involving the functionals $\cA$ and $\cB$ and suppose that these have minimizers $(\phi,\a)$ and $(m,w)$ respectively in the corresponding spaces as described above. The duality implies in particular that 
$$
\cA(\phi,\a)+\cB(m,w)=0.
$$
Now, consider $\eta\in\R$ small and $\d\in\T^d$ such that $|\d|$ is small and construct some special competitors by translations for the problem involving $\cB$, i.e. set
$$(m^{\eta,\d},w^{\eta,\d})(t,x):=(m(t+\eta,x+\d),w(t+\eta,x+\d)).$$
Let us remark that in order to preserve in particular the initial condition $m(0,\cdot)=m_0$, one needs to use some special cutoff functions in time, when one considers time translations, but let us skip this detail in this informal description.

Suppose that $F$ satisfies the strong coercivity condition \eqref{coercivity}. Then, the key part of the analysis is to show that 
\begin{equation}\label{ineq:error}
\cB(m^{\eta,\d},w^{\eta,\d})\le \cB(m,w)+ C(\eta^2+|\d|^2)
\end{equation}
for some constant $C>0$ independent of $\eta$ and $\d$. Then, one can show some $H^1_{{\rm{loc}}}$ estimates for $m$. Indeed, first we will have
$$\cA(\tilde\phi,\tilde\a)+\cB(\tilde m,\tilde w)\ge \|J(\tilde m)-J_*(\tilde\a)\|^2_{L^2}$$
for any competitors $(\tilde\phi,\tilde\a)$ and $(\tilde m,\tilde w)$ respectively. This means in particular by the duality that for the optimizers $(\phi,\a)$ and $(m,w)$ one has $J(m)=J_*(\a).$ Second, by these last two observations if we add the quantity $\cA(\phi,\a)$ to both sides of \eqref{ineq:error}, one gets
$$\|J(m)-J(m^{\eta,\d})\|^2_{L^2}\le C(\eta^2+|\d|^2),$$
thus the Sobolev estimate follows for $J(m)$.

\section{A warm-up: Sobolev estimates for the stationary problem}\label{sec:stationary}
Our goal, in fact, is to show that the weak solutions $(\phi,m,\l)$ given by Theorem \ref{theo:mainexergo} are such that for $\phi$ one can derive a second order Sobolev estimate, while in the same time one can obtain first order Sobolev estimates for $m$. These results are in the spirit of those in \cite{ProSan}.

We will make some extra assumptions:
\begin{equation}\label{eq:H5}
\begin{array}{ll}
F^*(\cdot, \a)\in C^{1,1}(\T^d),\ \forall\ \a\in \R, & D_xF^*(x,\cdot)\in{\rm{Lip_{loc}}}(\R),\\ [5pt]
D_\a F^*\in{\rm{Lip_{loc}}}(\T^d\times\R)\ \ {\rm{and}} &  H(\cdot,\xi) \in C^{1,1}(\T^d), \ \forall\ \xi\in\R^d.
\end{array}\tag{H5}
\end{equation}
For $h \in \bb{T}^d$, we define the translates
$$
\phi_h(x) := \phi(x-h).
$$
Choose $(\lambda,\phi) \in \overline{\s{K}_0}$ such that $\s{A}(\lambda,\phi)$ is finite.
Observe that
$$
\s{A}(\lambda,\phi_h) = \int_{\bb{T}^d} F^*(x,\lambda + H(x,\nabla \phi(x-h)))\dd x - \lambda = \int_{\bb{T}^d} F^*(x+h,\lambda + H(x+h,\nabla \phi(x)))\dd x - \lambda
$$
is also finite.
\begin{lemma}
	Suppose $(\lambda,\phi) \in \overline{\s{K}_0}$ is such that $\s{A}(\lambda,\phi)$ is finite.
	Let $A:\bb{T}^d \to \bb{R}$ be given by $A(h) := \s{A}(\lambda,\phi_h)$.
	Then $A$ is $C^{1,1}(\T^d)$.
\end{lemma}

\begin{proof}
	Observe that
	\begin{align*}
	A(y+h)-A(y) &= \int_{\bb{T}^d} [F^*(x+h,\lambda + H(x+h,\nabla \phi_y(x))) - F^*(x,\lambda + H(x,\nabla \phi_y(x)))]\dd x
	\\
	&= \int_{\bb{T}^d} D_xF^*(x,\lambda + H(x,\nabla \phi_y(x)))\cdot h\dd x\\ 
	&+ \int_{\T^d} D_\a F^*(x,\lambda + H(x,\nabla \phi_y(x))) D_x H(x,\nabla \phi_y(x))\cdot h \dd x + o(|h|).
	\end{align*}
	It follows that $A \in C^1(\T^d)$ with
	\begin{multline*}
	DA(h) = \int_{\bb{T}^d} [D_xF^*(x,\lambda + H(x,\nabla \phi_h(x))) + D_\a F^*(x,\lambda + H(x,\nabla \phi_h(x))) D_x H(x,\nabla \phi_h(x))] \dd x
	\\
	= \int_{\bb{T}^d} [D_xF^*(x+h,\lambda + H(x+h,\nabla \phi(x))) + D_\a F^*(x+h,\lambda + H(x+h,\nabla \phi(x))) D_x H(x+h,\nabla \phi(x))] \dd x.
	\end{multline*}
	By the assumption \eqref{eq:H5} we see that $DA(h)$ is Lipschitz continuous, as desired.
\end{proof}

\begin{lemma} \label{lem:h^2}
	Suppose $(\lambda,\phi) \in \overline{\s{K}_0}$ is such that $\s{A}(\lambda,\phi)$ is minimized.
	Then $A(h) = \s{A}(\lambda,\phi_h) = \s{A}(\lambda,\phi) + O(|h|^2)$.
\end{lemma}

\begin{proof}
	By the previous Lemma, we can expand $A(h) = A(0) + DA(0)h + O(|h|^2)$, and since $h = 0$ is a minimizer for $A$ we have $DA(0) = 0$.
\end{proof}

Now we are going to assume the following coercivity estimate on $F$: there exist $J,J_*:[0,+\infty)\to\R$ and $c_0>0$ such that 
\begin{equation}
\label{coercivity}
F(x,m) + F^*(x,a) - ma \geq c_0|J(m) - J_*(a)|^2.\tag{H6}
\end{equation}
Here the functions $J,J_*$ can be derived explicitly for many common examples of $F$ (see the examples in \cite[Section 3]{ProSan}).

We deduce that
\begin{align*}
\s{A}(\lambda,\phi) + \s{B}(m,w)
&= \int_{\bb{T}^d}  \left[F^*\left(x,\lambda+H(x,\nabla \phi(x)) \right) + m(x) H^*\left(x, -w(x)/m(x)\right)+ F(x,m(x))\right]\dd x - \lambda
\\
&\geq \int_{\bb{T}^d}  c_0|J(m)-J_*(\lambda + H(x,\nabla \phi(x)))|^2\dd x\\ 
&+ \int_{\T^d}m(x)\left[H(x,\nabla \phi(x)) + H^*\left(x, -w(x)/m(x)\right)\right]\dd x
\\
&\geq \int_{\bb{T}^d}  \left[c_0|J(m)-J_*(\lambda + H(x,\nabla \phi(x)))|^2 - \nabla \phi(x) \cdot w(x)\right]\dd x
\\
&= c_0\int_{\bb{T}^d} |J(m)-J_*(\lambda + H(x,\nabla \phi(x)))|^2\dd x
\end{align*}
for any $(\lambda,\phi) \in \overline{ \s{K}}_0$ and $(m,w) \in \overline{ \s{K}}_1$ such that $w \in L^{(q'r)'}(\bb{T}^d) = L^{r'q/(r'+q-1)}(\bb{T}^d)$
(note that $\nabla\cdot w = 0$ in the sense of distributions, hence $\int_{\bb{T}^d} \nabla \phi(x) \cdot w(x)\dd x = 0$).

In particular, suppose that $(\lambda,\phi) \in \overline{ \s{K}}_0$ and $(m,w) \in \overline{ \s{K}}_1$ are optimal, that is, $\s{A}(\lambda,\phi)$ and $\s{B}(m,w)$ are minimized.
Consider the translates $\phi_h(x) := \phi(x-h)$.
Then we have the estimates
\begin{align*}
\|J_*(H(\cdot,\nabla \phi_h) + \lambda)& - J_*(H(\cdot,\nabla \phi) + \lambda) \|_{L^2(\bb{T}^d)}^2 
\\
&\leq 2\|J_*(H(\cdot,\nabla \phi_h) + \lambda) - J(m) \|_{L^2(\bb{T}^d)}^2 + 2 \|J(m) - J_*(H(\cdot,\nabla \phi) + \lambda) \|_{L^2(\bb{T}^d)}^2 \\
&\leq \frac{2}{c_0}\left(\s{A}(\lambda,\phi_h) + \s{B}(m,w) \right) + \frac{2}{c_0} \left(\s{A}(\lambda,\phi) + \s{B}(m,w) \right) \\
&= \frac{2}{c_0}\left(\s{A}(\lambda,\phi_h) - \s{A}(\lambda,\phi) \right) = O(|h|^2),
\end{align*}
where in the last line we have used the duality result of Lemma \ref{Lem:dualiteergo} and the estimate from Lemma \ref{lem:h^2}.

This proves
\begin{theorem}\label{thm:stat_main}
	$J_*(H(\cdot,\nabla \phi) + \lambda)$ is in $H^1(\bb{T}^d)$.
\end{theorem}

Let us remark that the very same approach applies (even with simpler computations) to obtain Sobolev estimates for the $m$ variable. Indeed, let us assume that the following regularity conditions hold true.

\begin{equation}
F(\cdot, m)\in C^{1,1}(\T^d),\ \forall\ m\in[0,+\infty),\ \ {\rm{and}}\ \ H^*(\cdot,\zeta) \in C^{1,1}(\T^d), \ \forall\ \zeta\in\R^d.
\tag{H5'}
\end{equation}

Then, considering any competitor $(m,w)$ such that $\cB(m,w)$ is finite, for $h\in\T^d$ defining the translates $(m_h,w_h)$, it is easy to show that the map $\T^d\ni h\mapsto \cB(m_h,w_h)$ is $C^{1,1}$. A similar reasoning as for the proof of Theorem \ref{thm:stat_main} yields the estimate for $J(m)$. More precisely, one has 

\begin{theorem}\label{thm:stat_main2}
$J(m)$ is in $H^1(\T^d)$.
\end{theorem}

Observe also that formally Theorem \ref{thm:stat_main} implies Theorem \ref{thm:stat_main2}. In fact, since $\l+H(x,\nabla\phi)=f(x,m)$, Theorem \ref{thm:stat_main} implies that $J_*(f(\cdot,m))$ is in $H^1(\T^d)$. Since the coupling $f$ satisfying the growth condition \eqref{eq:coupling_growth}, $J$ grows like $|\cdot|^{q/2}$ and $J_*$ has a growth like $|\cdot|^{q'/2}$, from where $J_*(f(\cdot,m))$ has exactly the same growth as $J(m)$. However, this reasoning stays formal, one of the main obstructions being that $\l+H(x,\nabla\phi)=f(x,m)$ is an equality a.e.~only on $\{m>0\}$. Thus, in order to obtain the Sobolev estimate for $m$, one needs to use the regularity by duality machinery, as described above.

\section{Sobolev estimates for the time-dependent case}\label{sec:time}

Our goal in this section is to prove global in time Sobolev estimates for the time dependent system \eqref{eq:MFG_intro}.
We will see that the technique is quite different from the duality method used in the previous section.
The main reason for this is that we need to perturb the data, and not just the solution.
In this context, let us assume some extra monotonicity/coercivity conditions for $f$ and $H$. 

\vspace{10pt}

{\bf Additional assumptions}

\vspace{10pt}

\begin{enumerate}
	\item[(H7)] (Conditions on the coupling) Let $f$ be continuous on $\bb{T}^d \times (0,\infty)$, strictly increasing in the second variable, satisfying \eqref{eq:coupling_growth}.
	Moreover, we will assume that $f(x,m)$ is Lipschitz with respect to $x$, specifically
	\begin{equation}
	\label{f Lipschitz in x}
	|f(x,m) - f(y,m)| \leq Cm^{q-1}|x-y| \ \forall x,y \in \bb{T}^d, \ m \geq 0;
	\end{equation}
	and also that $f(x,m)$ is strongly monotone in $m$, i.e.~there exists $c_0 > 0$ such that
	\begin{equation} \label{f strongly monotone}
	\left( f(x,\tilde m) - f(x,m)\right)(\tilde m - m) \geq c_0\min\{\tilde m^{q-2},m^{q-2}\}|\tilde m-m|^2 \ \forall \tilde m, m \geq 0, \ \tilde m \neq m.
	\end{equation}
	(Note: if $q < 2$ one should interpret $0^{q-2}$ as $+\infty$ in \eqref{f strongly monotone}.
	In this way, when $\tilde m = 0$, for instance, \eqref{f strongly monotone} reduces to $f(x,m)m \geq c_0m^q$, as in the more regular case $q \geq 2$.)
	\item[(H8)] (Coercivity assumptions.) We assume that there exist $j_1,j_2:\bb{R}^d \to \bb{R}^d$ and $c_0>0$ such that
		\begin{equation} \label{eq:Hcoercivity}
		H(x,\xi) + H^*(x,\zeta) - \xi \cdot \zeta \geq c_0|j_1(\xi) - j_2(\zeta)|^2, \ \ \forall\ \xi,\zeta\in\R^d.
		\end{equation}
	Notice that with this assumption we require that a sharper version of Young's inequality (i.e. $H(x,\xi) + H^*(x,\zeta) - \xi \cdot \zeta \geq 0$) takes place. In particular, in light of \eqref{eq:hamiltonian_bounds},  one can check that if $H(x,\xi)=\frac{1}{r}|\xi|^r$ ($r>1$), then \eqref{eq:Hcoercivity} holds true with the choices of $j_1(\xi) := |\xi|^{r/2-1}\xi$, $j_2(\zeta) := |\zeta|^{r'/2-1}\zeta$ and $c_0$ can be computed explicitly depending only on $r$.
\end{enumerate}

\begin{remark}
	The following ``canonical" examples satisfy all of the hypotheses on the data:
	$$
	f(x,m) = c_1(x)m^{q-1}, \ \ H(x,\xi) = c_2(x)|\xi|^r,
	$$
	where $c_1,c_2$ are continuous functions bounded below by a positive constant, and $c_1$ is Lipschitz.
	In order to satisfy the hypotheses of Proposition \ref{prop:space-regularity}, we would also require $c_2$ to have $W^{2,\infty}$ regularity.
\end{remark}

\begin{remark}
	Hypothesis (H7) is akin to (H6), in that the monotonicity estimates on $f(x,\cdot)$ are directly related to the convexity of $F(x,\cdot)$.
	In particular, if $F(x,m) = \frac{1}{q}c_1(x)m^q$ with $c_1$ as described in the previous remark, then (H7) is satisfied (because $f(x,m) = c_1(x)m^{q-1}$) while (H6) is satisfied with $J(m) = m^{q/2}$.
	Then the result of Proposition \ref{prop:space-regularity} below can be interpreted as giving an $H^1$ estimate on $J(m)$.
\end{remark}

\subsection{Regularity in space}

Let $\delta \in \bb{T}^d$ be given.
Here we will consider perturbations of the data which are translations in the space variable.
To wit, we use the superscript $\delta$ to denote translation by $\delta$, i.e.~$\alpha^\delta(t,x) = \alpha(t,x+\delta),F^\delta(x,m) = F(x+\delta,m),$ and so on for all other data.
Accordingly, we will write
$$
\s{A}^\delta = \s{A}_{F^\delta, m_0^\delta}, \ \ \s{B}^\delta = \s{B}_{ H^\delta, F^\delta, \phi_T^\d}, \ \ \s{K}^\delta = \s{K}( \phi_T^\delta, H^\delta), \ \ \s{K}_1^\delta = \s{K}_1( m_0^\delta).
$$
Let $(\phi,\alpha) \in \s{K}$ be a minimizer for $\s{A}$ and $(m,w) \in \s{K}_1$ be a minimizer for $\s{B}$.
To construct minimizers for $\s{A}^\delta$ and $\s{B}^\delta$, we can use translations in space of $(\phi,\alpha)$ and $(m,w)$.
Indeed, we have that $(\phi^\delta,\alpha^\delta) \in \s{K}^\delta$ and that $\s{A}^\delta(\phi^\delta,\alpha^\delta) = \s{A}(\phi,\alpha)$.
One can then deduce that $(\phi^\delta,\alpha^\delta)$ is a minimizer for $\s{A}^\delta$, since any competitor can be translated backwards to get a competitor for $(\phi,\alpha)$, which is a minimizer for $\s{A}$.
In the same way $(m^\delta,w^\delta)(t,x) = (m,w)(t,x+\delta)$ is a minimizer for $\s{B}^\delta$.

We will use these facts in the following proof.

\begin{proposition} \label{prop:space-regularity}
	Let $m_0,\phi_T \in W^{2,\infty}(\bb{T}^d)$, and assume that $H$ is twice continuously differentiable in $x$ with
	\begin{equation} \label{eq:D_x^2 H}
	|D_x^2 H(x,\xi)| \leq C|\xi|^{r} + C.\tag{H9}
	\end{equation}

	Then $\|m^{\frac{q}{2} - 1}\nabla m\|_{L^2([0,T]\times\T^d)} \leq C$ and $\|m^{1/2}D (j_1(\nabla \phi))\|_{L^2([0,T]\times\T^d)} \leq C$.
\end{proposition}

\begin{proof}
	\emph{Step 1.}
	As above, let 
	$(m,w) \in \s{K}_1(m_0)$ be a minimizer for $\s{B} = \s{B}_{H,F,\phi_T}$. 
	Initially we take a minimizing sequence of smooth functions $\phi_n \in \s{K}_0(\phi_T)$ such that $\s{A}(\phi_n) \to \inf \s{A}$.
	Here $\s{A} = \s{A}_{F,m_0}$.
	Denoting the translates as above, we see that $\phi_n^\delta \in \s{K}_0(\phi_T^\delta)$ is also a minimizing sequence for $\s{A}^\delta = \s{A}_{F^\delta,m_0^\delta}$.
	
	Following the argument in \cite{CarGra} we get that, up to a subsequence, $\phi_n \to \phi$ in $L^\gamma([0,T]\times\T^d)$ with $1\leq \gamma<d/(d-1)$, $\nabla \phi_n \rightharpoonup \nabla \phi$ weakly in $L^r([0,T]\times\T^d;\R^d)$ as $n\to +\infty$, and for $\alpha_n := -\partial_t \phi_n + H(x,\nabla \phi_n)$ we have $\tilde \alpha_n := \alpha_n \chi_{\{{\alpha_n} \geq -l_n\}} \to \alpha$ in $L^1([0,T]\times\T^d)$ as $n\to +\infty$, for any given sequence $(l_n)_{n\in\mathbb{N}}$ such that $l_n \to \infty$ as $n\to+\infty$; moreover, $(\phi,\alpha) \in \s{K}$ is a minimizer of $\s{A}$.
	In particular, for all $M > 0$ we can pass to a subsequence on which $\alpha_n \chi_{\{\alpha_n \geq -M\}} \to \alpha \chi_{\{\alpha \geq -M\}}$ in $L^p([0,T]\times\T^d)$.
	
	 We now show that (up to a subsequence) $(\nabla \phi_n)$ converges weakly to $\nabla \phi$ in $L^r_m([0,T] \times \bb{T}^d;\bb{R}^d)$.
	To see this, first use $\phi_n$ as a test function in $\partial_t m + \nabla \cdot w = 0$ to get
	\begin{equation} \label{eq:phi_n-test-m}
	\int_{\bb{T}^d} \phi_T m(T) - \phi_n(0)m_0 = \int_0^T \int_{\bb{T}^d} (H(x,\nabla \phi_n)-\alpha_n)m + \nabla \phi_n \cdot w \dd x\dd t.
	\end{equation}
	Using \eqref{eq:hamiltonian_bounds} we get
	\begin{equation}
	\frac{1}{C}\int_0^T \int_{\bb{T}^d} |\nabla \phi_n|^r m
	\leq \|\phi_T\|_\infty + \int_{\bb{T}^d}|\phi_n(0)|m_0
	+ \int_0^T \int_{\bb{T}^d} (\alpha_n)_+ m + |\nabla \phi_n|m\left|\frac{w}{m}\right| \dd x\dd t,
	\end{equation}
	then apply Young's inequality to get
	\begin{equation} \label{eq:bound-on-phi_n}
	\frac{1}{C}\int_0^T \int_{\bb{T}^d} |\nabla \phi_n|^r m
	\leq \|\phi_T\|_\infty + \int_{\bb{T}^d}|\phi_n(0)|m_0
	+ \int_0^T \int_{\bb{T}^d} (\alpha_n)_+ m + Cm\left|\frac{w}{m}\right|^{r'} \dd x\dd t
	\end{equation}
	where, as usual, $C$ takes on a larger value.
	As argued in \cite{CarGraPorTon}, $\limsup_{n \to \infty} \int_{\bb{T}^d}|\phi_n(0)|m_0 \leq C$.
	Since $(\alpha_n)_+$ is bounded in $L^p$ and $m\left|\frac{w}{m}\right|^{r'}$ is integrable, this means that the right-hand side of \eqref{eq:bound-on-phi_n} is bounded by a constant, as desired.
	Thus, we can assume that (up to a subsequence) $(\nabla \phi_n)$ converges weakly to some $\xi$ in $L^r_m([0,T] \times \bb{T}^d;\bb{R}^d)$.
	
	We now claim that $\xi = \nabla \phi$, $m$-a.e.~in $[0,T]\times\T^d$.
	That is, we claim that $m\xi = m\nabla \phi$ a.e.
	It is enough to show that
	\begin{equation}
	\int_0^T \int_{\bb{T}^d} \psi \cdot \nabla \phi_n \ m \to \int_0^T \int_{\bb{T}^d} \psi \cdot \nabla \phi \ m
	\end{equation}
	for all $\psi \in C_c^\infty((0,T) \times \bb{T}^d; \bb{R}^d)$.
	Let $K > 0$ (large) and $\e > 0$ (small).
	Since $(\nabla\phi_n)$ converges weakly to $\nabla\phi$ in $L^r([0,T] \times \bb{T}^d;\bb{R}^d)$, we have
	\begin{equation} \label{eq:nabla phi_n convergence}
	\iint_{\{m \leq K\}} \psi \cdot \nabla \phi_n \ m \dd x \dd t \to \iint_{\{m \leq K\}} \psi \cdot \nabla \phi \ m \dd x \dd t.
	\end{equation}
	On the other hand, Young's inequality yields
	\begin{equation}\label{eq:zero}
	\left|\iint_{\{m {\geq} K\}} \psi \cdot \nabla \phi_n m \dd x \dd t\right|
	\leq \frac{\e^r}{r}\|\psi\|_\infty \int_0^T \int_{\bb{T}^d} |\nabla \phi_n|^r m \dd x \dd t
	+ \frac{\e^{-r'}}{r'}\|\psi\|_\infty \iint_{\{m {\geq} K\}} m \dd x \dd t
	\end{equation}
	so that, by \eqref{eq:bound-on-phi_n}, we have
	\begin{equation}
	\limsup_{n \to \infty} \left|\iint_{\{m {\geq} K\}} \psi \cdot \nabla \phi_n m \dd x \dd t\right|
	\leq C\frac{\e^r}{r}\|\psi\|_\infty + \frac{\e^{-r'}}{r'}\|\psi\|_\infty \iint_{\{m {\geq} K\}} m \dd x \dd t.
	\end{equation}
	Since $m$ is bounded in $L^q([0,T]\times\T^d)$, one has 
	$$\iint_{\{m\ge K\}}m\dd x\dd t\le \left(\iint_{\{m\ge K\}}1\dd x\dd t\right)^{1/q'}\|m\|_{L^q},$$
	thus Chebyshev's inequality yields that $\iint_{\{m\ge K\}}m\dd x\dd t\to 0,$ as $K\to+\infty.$ Thus, we let $K \to \infty$ and then $\e \to 0$ to get in \eqref{eq:zero}
	$$
	\limsup_{K \to \infty}\limsup_{n \to \infty} \left|\iint_{\{m {\geq} K\}} \psi \cdot \nabla \phi_n m \dd x \dd t\right| = 0,
	$$
	In fact, by similar arguments,
	$$
	\limsup_{K \to \infty}\limsup_{n \to \infty} \left|\iint_{\{m {\geq} K\}} \psi \cdot \nabla \phi_n \ m \dd x \dd t-\iint_{\{m {\geq} K\}} \psi \cdot \nabla \phi \ m \dd x \dd t\right| = 0,
	$$
	which, combined with \eqref{eq:nabla phi_n convergence}, proves the claim.	

	\emph{Step 2.}
	Now use $\phi_n^\delta$ and $\phi_n$ as test functions in $\partial_t m + \nabla \cdot w = 0$ and $\partial_t m^\delta + \nabla \cdot w^\delta = 0$ respectively to get
	\begin{equation} \label{eq:phi_n^delta-test-m}
	\int_{\bb{T}^d} \phi_T^\delta m(T) - \phi_n^\delta(0)m_0 = \int_0^T \int_{\bb{T}^d} (H(x+\delta,\nabla \phi_n^\delta)-\alpha_n^\delta)m + \nabla \phi_n^\delta \cdot w \dd x\dd t
	\end{equation}
	and
	\begin{equation} \label{eq:phi_n-test-m^delta}
	\int_{\bb{T}^d} \phi_T m^\delta(T) - \phi_n(0)m^\delta_0 = \int_0^T \int_{\bb{T}^d} (H(x,\nabla \phi_n)-\alpha_n)m^\delta + \nabla \phi_n \cdot w^\delta \dd x\dd t
	\end{equation}
	Combine \eqref{eq:phi_n^delta-test-m} with \eqref{eq:ibp_weak} and use the regularity of $H$ to get
	\begin{multline} \label{eq:space-regularity1}
	\int_{\bb{T}^d} (\phi_T^\delta - \phi_T) m(T) - (\phi^\delta_n(0)-\phi(0))m_0 
	\\
	=
	\int_0^T \int_{\bb{T}^d} (H(x+\delta,\nabla \phi^\delta_n) + H^*(x,-w/m) + \nabla \phi^\delta_n \cdot w/m - \alpha_n^\delta + f(m))m  \dd x\dd t
	\\
	=\int_0^T \int_{\bb{T}^d} (H(x,\nabla \phi^\delta_n) + H^*(x,-w/m) + \nabla \phi^\delta_n \cdot w/m - \alpha_n^\delta + f(m))m  \dd x\dd t	\\
	+ \int_0^T \int_{\bb{T}^d} \int_0^1 \langle D_x H(x + s\delta,\nabla \phi^\delta_n),\delta \rangle \dd s\dd x\dd t.
	\end{multline}
	Likewise, combine \eqref{eq:phi_n-test-m^delta} with \eqref{eq:ibp_weak} (but applied to translates) to get
	\begin{multline} \label{eq:space-regularity2}
	\int_{\bb{T}^d} (\phi_T - \phi_T^\delta) m^\delta(T) - (\phi_n(0)-\phi^\delta(0))m^\delta_0 
	\\
	=
	\int_0^T \int_{\bb{T}^d} (H(x,\nabla \phi_n) + H^*(x+\delta,-w^\delta/m^\delta) + \nabla \phi_n \cdot w^\delta/m^\delta - \alpha_n + f^\delta(m^\delta))m^\delta  \dd x\dd t
	\\
	=
	\int_0^T \int_{\bb{T}^d} (H(x-\delta,\nabla \phi^{-\delta}_n) + H^*(x,-w/m) + \nabla \phi_n^{-\delta} \cdot w/m - \alpha_n^{-\delta} + f(m))m  \dd x\dd t
	\\
	= \int_0^T \int_{\bb{T}^d} (H(x,\nabla \phi^{-\delta}_n) + H^*(x,-w/m) + \nabla \phi_n^{-\delta} \cdot w/m - \alpha_n^{-\delta} + f(m))m  \dd x\dd t
	\\
	- \int_0^T \int_{\bb{T}^d} \int_0^1 \langle D_x H(x - s\delta,\nabla \phi^{-\delta}_n),\delta \rangle \dd s\dd x\dd t.
	\end{multline}
	Note that, by the changes of variables $x \mapsto x + \delta$ for the first integral and $x \mapsto x - \delta$ for the second, followed by the translation $s \mapsto 1-s$, we get
	\begin{multline*}
	\int_0^T \int_{\bb{T}^d} \int_0^1 \langle D_x H(x + s\delta,\nabla \phi^{\delta}_n),\delta \rangle \dd s\dd x\dd t
	-
	\int_0^T \int_{\bb{T}^d} \int_0^1 \langle D_x H(x - s\delta,\nabla \phi^{-\delta}_n),\delta \rangle \dd s\dd x\dd t	
	\\
	= \int_0^T \int_{\bb{T}^d} \int_0^1 \langle D_x H(x + s\delta,\nabla \phi_n) - D_x H(x - s\delta,\nabla \phi_n),\delta \rangle \dd s\dd x\dd t
	\\
	= \int_0^T \int_{\bb{T}^d} \int_0^1 \int_{-s}^s \langle D_{xx}^2 H(x + r\delta,\nabla \phi_n)\delta,\delta \rangle \dd r\dd s\dd x\dd t.
	\end{multline*}
	Then adding together \eqref{eq:space-regularity1} and \eqref{eq:space-regularity2} we get
	\begin{multline} \label{eq:space-regularity3}
	\int_0^T \int_{\bb{T}^d} \left(H(x,\nabla \phi^\delta_n) +H(x,\nabla \phi^{-\delta}_n)+ 2H^*(x,-w/m) + \nabla \phi^\delta_n \cdot w/m+\nabla \phi^{-\delta}_n \cdot w/m\right)m  \dd x\dd t
	\\
	= \int_{\bb{T}^d} \left(\phi_T^\delta + \phi_T^{-\delta} - 2\phi_T\right)m(T)\dd x
	 -  \int_{\bb{T}^d} (\phi_n(0)(m_0^\delta + m_0^{-\delta}) - 2\phi(0)m_0)\dd x
	 \\
	+ \int_0^T \int_{\bb{T}^d} \left(\alpha_n^\delta + \alpha_n^{-\delta} - 2f(m)\right)m \dd x \dd t
	+ \int_0^T \int_{\bb{T}^d} \int_0^1 \int_{-s}^s \langle D_{xx}^2 H(x + r\delta,\nabla \phi_n)\delta,\delta \rangle \dd r\dd s\dd x\dd t.
	\end{multline}
Now we pass to the limit as $n\to+\infty$ in each term above. First, we since $H$ is convex in the second variable and $(\nabla \phi_n^{\pm\delta})$ converges weakly to $\nabla \phi^{\pm\delta}$ in $L^r_m([0,T]\times\T^d;\R^d)$, by weak lower semicontinuity we get
\begin{equation}\label{ineq:liminf1}
\int_0^T \int_{\bb{T}^d} \left(H(x,\nabla \phi^\delta) +H(x,\nabla \phi^{-\delta})\right)m  \dd x\dd t\le\liminf_{n\to+\infty}\int_0^T \int_{\bb{T}^d} \left(H(x,\nabla \phi^\delta_n) +H(x,\nabla \phi^{-\delta}_n)\right)m  \dd x\dd t.
\end{equation} 
Then, by the previous weak convergence and by the fact that $w/m\in L^{r'}_m([0,T]\times\T^d;\R^d)$, we have that 
\begin{equation}\label{ineq:liminf2}
\int_0^T \int_{\bb{T}^d} \left(\nabla \phi^\delta \cdot w/m+\nabla \phi^{-\delta} \cdot w/m\right)m  \dd x\dd t=\lim_{n\to+\infty}\int_0^T \int_{\bb{T}^d} \left(\nabla \phi^\delta_n \cdot w/m+\nabla \phi^{-\delta}_n \cdot w/m\right)m  \dd x\dd t
\end{equation}

Second, let us compute

	\begin{multline} \label{eq:space-regularity4}
	\liminf_{n \to \infty} \left\{\int_{\bb{T}^d} \left(\phi_T^\delta + \phi_T^{-\delta} - 2\phi_T\right)m(T)\dd x
	-  \int_{\bb{T}^d} \left(\phi_n(0)(m_0^\delta + m_0^{-\delta}) - 2\phi(0)m_0\right)\dd x \right.
	\\
	+ \left. \int_0^T \int_{\bb{T}^d} \left(\alpha_n^\delta + \alpha_n^{-\delta} - 2f(m)\right)m \dd x \dd t
	+ \int_0^T \int_{\bb{T}^d} \int_0^1 \int_{-s}^s \langle D_{xx}^2 H(x + r\delta,\nabla \phi_n)\delta,\delta \rangle \dd r\dd s\dd x\dd t \right\}.\\
	\leq \limsup_{n \to \infty} \left\{\int_{\bb{T}^d} \left(\phi_T^\delta + \phi_T^{-\delta} - 2\phi_T\right)m(T)\dd x
	-  \int_{\bb{T}^d} \left(\phi_n(0)(m_0^\delta + m_0^{-\delta}) - 2\phi(0)m_0\right)\dd x \right.
	\\
	+ \left. \int_0^T \int_{\bb{T}^d} \left(\alpha_n^\delta + \alpha_n^{-\delta} - 2f(m)\right)m \dd x \dd t
	+ \int_0^T \int_{\bb{T}^d} \int_0^1 \int_{-s}^s \langle D_{xx}^2 H(x + r\delta,\nabla \phi_n)\delta,\delta \rangle \dd r\dd s\dd x\dd t \right\}.
	\end{multline}
Let us recall that $(\phi_n)_{n\ge 0}$ converges in $L^1([0,T]\times\T^d)$ to $\phi\in BV([0,T]\times\T^d)$ and $(\phi_n(0))_{n\ge 0}$ is bounded in $L^1(\T^d)$. In particular $\partial_t\phi_n\weakly\partial_t\phi$, as $n\to+\infty$, weakly-$\star$ in $\sM([0,T]\times\T^d).$
These imply furthermore that 
\begin{equation}\label{ineq:liminf3}
\limsup_{n\to+\infty}-  \int_{\bb{T}^d} \phi_n(0)(m_0^\delta + m_0^{-\delta}) \dd x\le -  \int_{\bb{T}^d} \phi(0)(m_0^\delta + m_0^{-\delta})\dd x
\end{equation}
Indeed, on the one hand we have 
\begin{align*}
-\int_{\T^d}\phi_n(0,x)\psi(x)\dd x&=-\int_{\T^d}\phi_T(x)\psi(x)\dd x+\int_0^T\int_{\T^d}\partial_t\phi_n\psi\dd x\dd t\\
&\to-\int_{\T^d}\phi_T(x)\psi(x)\dd x+\int_0^T\int_{\T^d}\partial_t\phi\psi\dd x\dd t\\
&=\int_{\T^d}(\phi(T,x)-\phi_T(x))\psi(x)\dd x-\int_{\T^d}\phi(0,x)\psi(x)\dd x
\end{align*}
as $n\to+\infty$ for all $\psi\in C(\T^d)$.
On the other hand, since $\phi(T,x)\le\phi_T(x)$ for a.e. $x\in\T^d$, for $\psi\ge 0$ one can conclude that
$$\limsup_{n\to+\infty}-  \int_{\bb{T}^d} \phi_n(0)\psi \dd x\le -  \int_{\bb{T}^d} \phi(0)\psi\dd x,$$
which implies \eqref{ineq:liminf3} as desired.

	As for the third term, we take an arbitrary $M > 0$ and get
	\begin{align*}
	\limsup_{n\to \infty} \int_0^T \int_{\bb{T}^d} &\left(\alpha_n^\delta + \alpha_n^{-\delta} - 2f(m)\right)m \dd x \dd t\\
&\leq \limsup_{n\to \infty} \int_0^T \int_{\bb{T}^d} \left(\alpha_n^\delta \chi_{\{\alpha_n^\delta \geq -M\}} + \alpha_n^{-\delta}\chi_{\{\alpha_n^{-\delta} \geq -M\}} - 2f(m)\right)m \dd x \dd t
	\\
	&= \int_0^T \int_{\bb{T}^d} \left(f^\delta(m^\delta)\chi_{\{f^\delta(m^\delta) \geq -M\}}  + f^{-\delta}(m^{-\delta})\chi_{\{f^{-\delta}(m^{-\delta}) \geq -M\}} - 2f(m)\right)m \dd x \dd t,
	\end{align*}
	where we used the fact (see for instance \cite[Theorem 3.5-(i)]{CarGra}) that since $(\phi,\alpha)$ is a minimizer of the primal problem and $(m,w)$ is a minimizer of the dual problem, one has that $\alpha=f(\cdot,m)$ a.e.  This, after letting $M \to \infty$ becomes
	$$
	\limsup_{n\to \infty} \int_0^T \int_{\bb{T}^d} \left(\alpha_n^\delta + \alpha_n^{-\delta} - 2f(m)\right)m \dd x \dd t
	\leq 
	\int_0^T \int_{\bb{T}^d} \left(f^\delta(m^\delta)  + f^{-\delta}(m^{-\delta}) - 2f(m)\right)m \dd x \dd t.
	$$
	Lastly,  by \eqref{eq:D_x^2 H} we can assert
	\begin{equation} \label{eq:space-regularity8}
	\left|\int_0^T \int_{\bb{T}^d} \int_0^1 \int_{-s}^s \langle D_{xx}^2 H(x + r\delta,\nabla \phi_n)\delta,\delta \rangle \dd r\dd s\dd x\dd t \right|
	\leq C(\|\nabla \phi_n\|_{r}^{r} + 1)|\delta|^2 \leq C|\delta|^2.
	\end{equation}
Thus, when passing to the $\liminf$ in \eqref{eq:space-regularity3} as $n\to+\infty$, the previous inequalities and estimations imply, 	

\begin{multline} \label{eq:space-regularity3-limit}
	\int_0^T \int_{\bb{T}^d} \left(H(x,\nabla \phi^\delta) + H^*(x,-w/m) + \nabla \phi^\delta \cdot w/m\right)m  \dd x\dd t\\ 
	+\int_0^T\int_{\T^d}\left(H(x,\nabla \phi^{-\delta})+ H^*(x,-w/m) +\nabla \phi^{-\delta} \cdot w/m\right)m  \dd x\dd t \\
	\le \int_{\bb{T}^d} \left(\phi_T^\delta + \phi_T^{-\delta} - 2\phi_T\right)m(T)\dd x
	 -  \int_{\bb{T}^d} \phi(0)(m_0^\delta + m_0^{-\delta} - 2m_0)\dd x
	 \\
	+ \int_0^T \int_{\bb{T}^d} \left(f^\delta(m^\delta)  + f^{-\delta}(m^{-\delta}) - 2f(m)\right)m \dd x \dd t
	+ C|\delta|^2
	\end{multline}

	Now, the assumption \eqref{eq:Hcoercivity} on $H$ together with \eqref{eq:space-regularity3-limit} on the one hand and the inequality $|a+b|^2 \leq 2(|a|^2+|b|^2),\ \forall a,b\in\R^d$ on the other hand yield
	\begin{multline} \label{eq:space-regularity5}
	\frac{c_0}{2}\int_0^T \int_{\bb{T}^d} \left(|j_1(\nabla \phi^\delta) - j_1(\nabla \phi^{-\delta})|^2\right) m \dd x \dd t
	\\
	\leq \int_{\bb{T}^d} \left(\phi_T^\delta + \phi_T^{-\delta} - 2\phi_T\right)m(T)\dd x
	-  \int_{\bb{T}^d} \phi(0)\left(m_0^\delta + m_0^{-\delta} - 2m_0\right)\dd x 
	\\
	+  \int_0^T \int_{\bb{T}^d} \left(f^\delta(m^\delta)  + f^{-\delta}(m^{-\delta}) - 2f(m)\right)m \dd x \dd t
	+ C|\delta|^2.
	\end{multline}

	\emph{Step 3.}
	We estimate the terms in \eqref{eq:space-regularity5}.
	First, we have
	\begin{equation} \label{eq:space-regularity6}
	\int_{\bb{T}^d} \phi(0)\left(2m_0 - m_0^\delta-m_0^{-\delta}\right)\dd x
	\leq C|\delta|^2\int_{\bb{T}^d} |\phi(0)|\dd x,
	\end{equation}
	where $C$ depends on $\|m_0\|_{W^{2,\infty}}$, and we know that $\int_{\bb{T}^d} |\phi(0)|\dd x$ is well-defined and finite (see \cite[Lemma 5.1]{CarGraPorTon}).
	Similarly,
	\begin{equation} \label{eq:space-regularity7}
	\int_{\bb{T}^d} m(T)\left(2\phi_T - \phi_T^\delta-\phi_T^{-\delta}\right)\dd x
	\leq C|\delta|^2\int_{\bb{T}^d} m(T)\dd x=C|\delta|^2,
	\end{equation}
	where $C$ depends on $\|\phi_T\|_{W^{2,\infty}}$.
	There is one more term to estimate.
	We rewrite it using change of variables and then apply Equations \eqref{f strongly monotone} and \eqref{f Lipschitz in x} to get
	\begin{multline} \label{eq:space-regularity9}
	\int_0^T \int_{\bb{T}^d} (f^\delta(m^\delta)  + f^{-\delta}(m^{-\delta}) - 2f(m))m \dd x \dd t
	= -\int_0^T \int_{\bb{T}^d} (f^\delta(m^\delta) - f(m))(m^{\delta} - m) \dd x \dd t
	\\
	= - \iint_{\{m^\delta \leq m\}} (f^\delta(m^\delta) - f(m^\delta))(m^{\delta} - m) \dd x \dd t 
	- \iint_{\{m^\delta \leq m\}} (f(m^\delta) - f(m))(m^{\delta} - m) \dd x \dd t
	\\
	- \iint_{\{m < m^\delta\}} (f^\delta(m^\delta) - f^\delta(m))(m^{\delta} - m) \dd x \dd t
	- \iint_{\{m < m^\delta\}} (f^\delta(m) - f(m))(m^{\delta} - m) \dd x \dd t
	\\
	\leq C\int_0^T \int_{\bb{T}^d} |\delta|\min\{(m^\delta)^{q-1},m^{q-1}\}|m^{\delta} - m| \dd x \dd t 
	- c_0\int_0^T \int_{\bb{T}^d} \min\{(m^\delta)^{q-2},m^{q-2}\}|m^\delta - m|^2 \dd x \dd t
	\\
	\leq C|\delta|^2 \int_0^T \int_{\bb{T}^d} \min\{m^\delta,m\}^{q} \dd x \dd t 
	- \frac{c_0}{2}\int_0^T \int_{\bb{T}^d} \min\{(m^\delta)^{q-2},m^{q-2}\}|m^\delta - m|^2 \dd x \dd t,
	\end{multline}
	where, we used Young's inequality in the last inequality, and the expression $\min\{(m^\delta)^{q-2},m^{q-2}\}|m^\delta - m|^2$ is  treated as zero whenever $m^\delta = m$ (even in the case $q < 2$).
	Since
	$$
	\int_0^T \int_{\bb{T}^d} \min\{m^\delta,m\}^{q} \dd x \dd t \leq \int_0^T \int_{\bb{T}^d} m^q \dd x \dd t  \leq C,
	$$
	we see from plugging \eqref{eq:space-regularity6}, \eqref{eq:space-regularity7}, and \eqref{eq:space-regularity9} into \eqref{eq:space-regularity5} that
	\begin{multline} 
	\frac{c_0}{2}\int_0^T \int_{\bb{T}^d} \left(|j_1(\nabla \phi^\delta) - j_1(\nabla \phi^{-\delta})|^2\right) m \dd x \dd t
	\\
	+ \frac{c_0}{2}\int_0^T \int_{\bb{T}^d} \min\{(m^\delta)^{q-2},m^{q-2}\}|m^\delta - m|^2 \dd x \dd t
	\leq C|\delta|^2,
	\end{multline}
	where $C$ depends only on the data.	
	The result now follows from dividing both sides by $|\delta|^2$ and letting $\delta \to 0$.
\end{proof}

As a potential application of Proposition \ref{prop:space-regularity}, one might hope to derive second order in space regularity for $\phi$.
For this, one would like some summability estimate on $m^{-1}$.
Since, thanks to \cite{CarGra}, the growth of $f(x,m)$ is unrestricted for small values of $m$, all of our results are still valid for data satisfying, for example, $F(x,m) \geq \frac{1}{C}m^{-s}- C$ for some $s \geq 1$, thus providing the desired estimate.
Such a case would correspond to extreme congestion penalization, in which players experience a great benefit by moving toward unoccupied spaces.
In this context, the following corollary is meaningful.

\begin{corollary}
	Suppose $m^{-1} \in L^s([0,T]\times\T^d)$ for some $s \geq 1$.
	Then 
	\begin{equation} \label{eq:space-regularity-j_1(nabla phi_n)}
	\|\nabla (j_1(\nabla \phi))\|_{\frac{2s}{s+1}}^2 \leq C\|m^{-1}\|_{s},
	\end{equation}
	where, for simplicity of writing $\|\cdot\|_{\eta}$ denotes the norm $\|\cdot\|_{L^\eta([0,T]\times\T^d)}$.
\end{corollary}

\begin{proof}
	Observe that by H\"older's inequality and Proposition \ref{prop:space-regularity}, we have
	\begin{equation}
	\|j_1(\nabla \phi^\delta) - j_1(\nabla \phi^{-\delta})\|_{\frac{2s}{s+1}}^2 
	\leq \|m^{-1}\|_{s}\int_0^T \int_{\bb{T}^d} (|j_1(\nabla \phi^\delta) - j_1(\nabla \phi^{-\delta})|^2) m \dd x \dd t\le C\|m^{-1}\|_{s}|\d|^2,
	\end{equation}
which implies the thesis of this corollary dividing both sides by $|\d|^2$ and taking $\d\to 0$.	
\end{proof}

\subsection{Some time regularity}

Assume that $r = 2$, so that the Hamiltonian is of quadratic growth.
In line with this assumption, we will assume that
\begin{equation} \label{eq:D^2 H bounded}
|D_{pp}^2 H(x,\xi)| \leq C, \ |D_{xp}^2 H(x,\xi)| \leq C(|\xi| + 1) \ \ \forall x \in \bb{T}^d, \forall \xi \in \bb{R}^d.\tag{H10}
\end{equation}
Then the result of Proposition \ref{prop:space-regularity} reads
\begin{equation} \label{eq:space-regularity-r=2}
\|m^{q/2 -1}\nabla m\|_{L^2([0,T]\times\T^d)}, \|m^{1/2}D^2 \phi\|_{L^2([0,T]\times\T^d)} \leq C.
\end{equation}

Recalling that $w = -mD_\xi H(x,\nabla\phi)$, we compute
\begin{equation}
-\nabla \cdot w = \nabla m \cdot D_\xi H(x,\nabla\phi) + m \ \mathrm{tr}  (D_{xp}^2 H(x,\nabla \phi)) + m \ \mathrm{tr}(D_{pp}^2 H(x,\nabla \phi)D^2 \phi)
\end{equation}
and deduce from \eqref{eq:D^2 H bounded} that
\begin{equation}
m^{q/2 -1}|\nabla \cdot w| \leq Cm^{q/2 -1}|\nabla m|(|\nabla \phi| + 1) + Cm^{q/2}(|\nabla \phi| + 1) + Cm^{q/2 -1/2}m^{1/2}|D^2 \phi|.
\end{equation}
Now since $\|H(x,\nabla \phi)\|_{L^1} \leq C$, we have that $|\nabla \phi| \in L^2([0,T]\times\T^d)$ by \eqref{eq:hamiltonian_bounds}.
Moreover, since $q > 1$ and $\|m\|_{L^q} \leq C$ it follows (by H\"older's inequality) that $\|m^{q/2 -1/2}\|_{L^2} \leq C$.
Therefore \eqref{eq:space-regularity-r=2} implies, by H\"older's inequality, that
\begin{equation} \label{eq:div w regularity}
\|m^{q/2 -1}\nabla \cdot w\|_{L^1([0,T]\times\T^d)} \leq C.
\end{equation}
We deduce the following:
\begin{corollary} [Global $L^1$ bounds on $\partial_t m$]\label{cor:time-reg}
	Assume $r = 2$ and that \eqref{eq:D^2 H bounded} holds.
	Suppose $(m,w)$ is the minimizer of $\s{B}$.
	Then there exist $C>0$ depending only on the data such that 
	$$\|\partial_t (m^{q/2})\|_{L^1([0,T]\times\T^d)} \leq C.$$
\end{corollary}

\begin{proof}
	Apply \eqref{eq:div w regularity} to the equation $\partial_t m = -\nabla \cdot w$.
\end{proof}

The argument breaks down if $r \neq 2$.
For simplicity, set $H(x,\xi) = \frac{1}{r}|\xi|^r$.
The result of Proposition \ref{prop:space-regularity} reads
\begin{equation} \label{eq:space-regularity-r neq 2}
\|m^{q/2 -1}\nabla m\|_{L^2}, \|m^{1/2}{(|\nabla\phi|^{r/2-1}D^2 \phi+(r/2-1)|\nabla\phi|^{r/2-2}D^2\phi\nabla\phi\otimes\nabla\phi)}\|_{L^2} \leq C.
\end{equation}
We compute
\begin{equation}
-\nabla \cdot w = \nabla m \cdot |\nabla \phi|^{r-2}\nabla \phi + m |\nabla \phi|^{r-2}\Delta \phi+{m(r-2)|\nabla\phi|^{r-4}(D^2\phi\nabla\phi)\cdot\nabla\phi}.
\end{equation}
If $r < 2$, then the last two terms of the r.h.s. are degenerate: it is possible that $m|\nabla \phi|^{r-2}\Delta \phi$ is not integrable, since even though one would have $L^2$ regularity for $|\nabla \phi|^{r/2-1}D^2 \phi$, there are no estimates on the remaining factor of $m|\nabla \phi|^{r/2-1}$.
On the other hand, if $r > 2$, then the first term on the r.h.s. is problematic: $|\nabla \phi|^{r-2}\nabla \phi \in L^{r'}([0,T]\times\T^d)$, but $r' < 2$, and thus $L^2$-regularity for $\nabla (m^{q/2})$ is not enough to ensure that the product $\nabla (m^{q/2}) \cdot |\nabla \phi|^{r-2}\nabla \phi$ is integrable.

\vspace{15pt}

{\sc Acknowledgements} 

\vspace{15pt}

We thank Filippo Santambrogio for the fruitful discussions during this project. In particular, some discussions with him and his paper \cite{San17} inspired us in writing this paper. We thank also Pierre Cardaliaguet for his remarks and valuable suggestions on the manuscript. 

The first author was supported by the National Science Foundation through Grant DMS-1612880.

The second author was partially supported by the Gaspar Monge Program for Optimization and Operation Research (PGMO) via the project {\it VarMFGPDE}. He would like to thank the hospitality of Tohoku University and Tohoku Forum for Creativity, Sendai, Japan in the framework of the Thematic Program on {\it Nonlinear Partial Differential Equations for Future Applications} in Summer 2017.

\bibliographystyle{alpha}
\bibliography{jameson_alpar}{}
\end{document}